\documentclass[twoside,leqno,10pt]{amsart}
\usepackage{amsfonts}
\usepackage{amsmath}
\usepackage{amscd}
\usepackage{amssymb}
\usepackage{amsthm}
\usepackage{amsrefs}
\usepackage{latexsym}
\usepackage{mathrsfs}
\usepackage{bbm}
\usepackage{enumerate}
\usepackage{color}
\begin{document}

\newtheorem{theorem}[subsection]{Theorem}
\newtheorem{proposition}[subsection]{Proposition}
\newtheorem{lemma}[subsection]{Lemma}
\newtheorem{corollary}[subsection]{Corollary}
\newtheorem{conjecture}[subsection]{Conjecture}
\newtheorem{prop}[subsection]{Proposition}
\numberwithin{equation}{section}
\newcommand{\mr}{\ensuremath{\mathbb R}}
\newcommand{\mc}{\ensuremath{\mathbb C}}
\newcommand{\dif}{\mathrm{d}}
\newcommand{\intz}{\mathbb{Z}}
\newcommand{\ratq}{\mathbb{Q}}
\newcommand{\natn}{\mathbb{N}}
\newcommand{\comc}{\mathbb{C}}
\newcommand{\rear}{\mathbb{R}}
\newcommand{\prip}{\mathbb{P}}
\newcommand{\uph}{\mathbb{H}}
\newcommand{\fief}{\mathbb{F}}
\newcommand{\majorarc}{\mathfrak{M}}
\newcommand{\minorarc}{\mathfrak{m}}
\newcommand{\sings}{\mathfrak{S}}
\newcommand{\fA}{\ensuremath{\mathfrak A}}
\newcommand{\mn}{\ensuremath{\mathbb N}}
\newcommand{\mq}{\ensuremath{\mathbb Q}}
\newcommand{\half}{\tfrac{1}{2}}
\newcommand{\f}{f\times \chi}
\newcommand{\summ}{\mathop{{\sum}^{\star}}}
\newcommand{\chiq}{\chi \bmod q}
\newcommand{\chidb}{\chi \bmod db}
\newcommand{\chid}{\chi \bmod d}
\newcommand{\sym}{\text{sym}^2}
\newcommand{\hhalf}{\tfrac{1}{2}}
\newcommand{\sumstar}{\sideset{}{^*}\sum}
\newcommand{\sumprime}{\sideset{}{'}\sum}
\newcommand{\sumprimeprime}{\sideset{}{''}\sum}
\newcommand{\shortmod}{\ensuremath{\negthickspace \negthickspace \negthickspace \pmod}}
\newcommand{\V}{V\left(\frac{nm}{q^2}\right)}
\newcommand{\sumi}{\mathop{{\sum}^{\dagger}}}
\newcommand{\mz}{\ensuremath{\mathbb Z}}
\newcommand{\leg}[2]{\left(\frac{#1}{#2}\right)}
\newcommand{\muK}{\mu_{\omega}}
\newcommand{\sumflat}{\sideset{}{^\flat}\sum}
\newcommand{\thalf}{\tfrac12}
\newcommand{\lam}{\lambda}

\title[Mean squares of real character sums]{Mean squares of quadratic twists of the M\"obius function}

%%\date{\today}
\author[P. Gao]{Peng Gao}
\address{School of Mathematical Sciences, Beihang University, Beijing 100191, China}
\email{penggao@buaa.edu.cn}

\author[L. Zhao]{Liangyi Zhao}
\address{School of Mathematics and Statistics, University of New South Wales, Sydney NSW 2052, Australia}
\email{l.zhao@unsw.edu.au}

\begin{abstract}
In this paper, we evaluate asymptotically the sum
\[ \sum_{d \leq X} \left( \sum_{n \leq Y} \mu(n)\leg {8d}{n} \right)^2, \]
where $\leg {8d}{n}$ is the Kronecker symbol and $d$ runs over positive, odd, square-free integers.
\end{abstract}

\maketitle

\noindent {\bf Mathematics Subject Classification (2010)}: 11N37, 11L05, 11L40 \newline

\noindent {\bf Keywords}: mean square, quadratic Dirichlet character, M\"obius function
\section{Introduction}

   Let $\mu$ denote the M\"obius function and the corresponding Mertens function $M(x)$ is defined to be
\begin{align*}
%%\label{Jksim}
 M(x) = \sum_{n \leq x}\mu(n).
\end{align*}

The size of $M(x)$ is inextricably connected with the Riemann hypothesis (RH).  It is known (see \cite{Sound09-1}) that RH is equivalent to
\begin{align}
\label{Mbound}
  M(x) \ll x^{1/2+\varepsilon},
\end{align}
  for any $\varepsilon>0$. \newline

There have been a number of subsequent refinements of the bounds in \eqref{Mbound}, all under RH.  In \cite{Landau24}, E. Landau proved that
\eqref{Mbound} is valid with $\varepsilon \ll \log \log \log x/\log \log x$.  This bound was improved to $\varepsilon \ll 1/\log \log x$ by
E. C. Titchmarsh \cite{Titchmarsh27},  to $\varepsilon \ll (\log x)^{-22/61}$ by H. Maier and H. L. Montgomery \cite{MM09} and by
K. Soundararajan \cite{Sound09-1}
to
\begin{align*}
%%\label{Mbound1}
  M(x) \ll x^{1/2}\exp \big ((\log x)^{1/2}(\log \log \log x)^{14} \big ).
\end{align*}
The power $(\log \log \log x)^{14}$ in the above expression has been improved to $(\log \log \log x)^{5/2+\varepsilon}$ for any $\varepsilon >0$ by
M. Balazard and A. de Roton in \cite{BR08} upon refining the method of Soundararajan. \newline

 One may consider more generally the sum with the M\"obius function twisted by a Dirichlet character $\chi$ modulo $q$. More precisely,
we define
\begin{align*}
M(x, \chi) = \sum_{n \leq x}\mu(n)\chi(n).
\end{align*}
Similar to the relation between $M(x)$ and RH, the size of $M(x, \chi)$ is related to the generalized Riemann hypothesis (GRH) of the corresponding Dirichlet
$L$-function $L(s, \chi)$. It follows from Perron’s formula that GRH implies that
\begin{align}
\label{Mchibound}
M(x, \chi) \ll x^{1/2+\varepsilon},
\end{align}
for any $\varepsilon >0$.  Conversely, \eqref{Mchibound} gives, via partial summation, the convergence of the Dirichlet series
of $1/L(s, \chi)$ for any $s > 1/2$, and therefore GRH for $L(s, \chi)$.  While studying
sums of the M\"obius function in arithmetic progressions, L. Ye \cite{Ye} established that under
GRH, uniformly for $q$ and $x$,
\begin{align*}
%%\label{Mchibound1}
M(x, \chi) \ll x^{1/2}\exp \big ((\log x)^{3/5+o(1)}\big ).
\end{align*}
  This improved an earlier result of K. Halupczok and B. Suger \cite[Lemma 1,2]{H&S13}.  Moreover, it follows from a general result of H. Maier and A. Sankaranarayanan \cite{M&S16} on multiplicative M\"obius-like functions that $\varepsilon \ll 1/\log \log x$ in \eqref{Mchibound} under GRH, which is comparable to the above mentioned result of Titchmarsh \cite{Titchmarsh27} on $M(x)$. \newline

  As noted in \cite{MM09}, the behavior of $M(x)$  depends both on the distribution
of $|\zeta'( \rho)|$ as $\rho = 1/2 + i\gamma$ runs over the non-trivial zeros of the Riemann zeta function $\zeta(s)$ (under RH), and on the linear independence of the $\gamma$. This makes it difficult to predict the behavior of $M(x)$ in any finer way. For example, a well-known conjecture of Mertens claiming that $|M(x)| \leq \sqrt{x}$ was disproved by A. M. Odlyzko and H. J. J. te Riele \cite{O&R}. In connection to this, one also has the weak Mertens conjecture which asserts that
\begin{align*}
%%\label{Mchibound1}
 \int\limits^X_2 \Big ( \frac {M(x)}{x} \Big )^2 \dif x \ll \log X.
\end{align*}

  In \cite[Theorem 3]{Ng04}, N. Ng proved that as $X \rightarrow \infty$, for some constant $c$,
\begin{align}
\label{Msquareest}
 \int\limits^X_2 \Big ( \frac {M(x)}{x} \Big )^2 \dif x \sim c\log X,
\end{align}
  provided that one assumes RH and
\begin{align*}
  \sum_{0<\Im(\rho)\le T}|\zeta'(\rho)|^{-2} \ll T.
\end{align*}

  One may interpret \eqref{Msquareest} as a mean square type of estimation for $M(x)$ and in this situation one is able to evaluate the average asymptotically. We are thus motivated to seek for other mean square estimations involving the M\"obius function and it is the aim of our paper to study one such case. \newline

  To state our result, we write $\chi_{d}$ for the Kronecker symbol $\leg {d}{\cdot}$ and note that if $d$ is odd and square-free, $\chi_{8d}$ is a primitive Dirichlet character.  We are interested in the following sum
\begin{align*}
%%\label{SXY}
    S(X,Y)=\sumstar_{\substack {0< d \leq X}}  M(Y, \chi_{8d} )^2,
\end{align*}
   where the asterisk indicates that $d$ runs over odd and square-free integers. \newline

We may view $S(X,Y)$ as a mean square expression involving $M(Y, \chi_{8d})$ and one expects an asymptotic expression for it.  In fact, it is not difficult to obtain one if $Y^2<X$ using the P\'olya-Vinogradov inequality to control the contribution of the off-diagonal terms.  The situation is more intriguing for larger $Y$'s, especially if $X$ and $Y$ are of comparable size.  For instance, the sum
\begin{align*}
  \sum_{\substack {m \leq X \\ (m, 2)=1}}\sum_{\substack {n \leq Y \\ (n, 2)=1}} \leg {m}{n}
\end{align*}
 can be evaluated asymptotically if $Y=o(X/\log X)$ or $X=o(Y/\log Y)$ using the P\'olya-Vinogradov inequality. In \cite{CFS}, J. B. Conrey, D. W. Farmer and K. Soundararajan applied a Poisson summation formula developed by Soundararajan in \cite{sound1} to obtain an asymptotic formula for the other ranges.  We also note here that extensions and generalizations of this problem were studied by the authors in \cites{G&Zhao2019, G&Zhao2020, G&Zhao2022}. \newline

In studying $S(X,Y)$, we shall also utlize the Poisson summation formula given in \cite{sound1} as well as the techniques developed by  K. Soundararajan and M. P. Young \cite{S&Y} in their work on the second moment of quadratic twists of modular $L$-functions.  For technical reasons, we consider smoothed sums instead. We thus fix two non-negative, smooth functions $\Phi(x), W(x)$ that are compactly supported on ${\mr}_+=(0,\infty)$.  Set
\begin{align} \label{SXYPW}
    S(X,Y; \Phi, W)=\sumstar_{\substack {d}} \Big(\sum_{n} \mu(n)\chi_{8d}(n)\Phi \Big(\frac nX \Big)\Big)^2 W \Big(\frac d{X} \Big).
\end{align}

  We shall evaluate $S(X,Y; \Phi, W)$ asymptotically as follows.
\begin{theorem}
\label{meansquare}
   With the notation as above and assuming the truth of GRH, for large $X$ and $Y$, we have
\begin{align}
\label{S}
\begin{split}
   S(X,Y; \Phi, W)=& \frac{4}{\pi^2}  XY \widetilde{h}_1(1,1) Z_2(1)  + O\left( X^{1/2+\varepsilon}Y^{3/2+\varepsilon}+XY^{1/2+\varepsilon} \right),
\end{split}
\end{align}
where $\widetilde{h}_1(1,1)$ is given in \eqref{h1half} and the function $Z_2(u)$ is defined in \eqref{eq:Z(u,v)}.
\end{theorem}

   One checks that \eqref{S} gives a valid asymptotic formula if $Y \ll X^{1-\varepsilon}$ for any $\varepsilon>0$.

\section{Preliminaries}
\label{sec 2}

  We gather first a few auxiliary results necessary in the proof of Theorem \ref{meansquare} in this section.
\subsection{Gauss sums}
\label{section:Gauss}

  For all odd integers $k$ and all integers $m$, define the Gauss-type sums $G_m(k)$, as in \cite[Sect. 2.2]{sound1},
\begin{align}
\label{G}
    G_m(k)=
    \left( \frac {1-i}{2}+\left( \frac {-1}{k} \right)\frac {1+i}{2}\right)\sum_{a \shortmod{k}}\left( \frac {a}{k} \right) e \left( \frac {am}{k} \right), \quad \mbox{where} \quad e(x) = \exp (2 \pi i x) .
\end{align}
Let $\varphi(m)$ be the Euler totient of $m$. Our next result is taken from \cite[Lemma 2.3]{sound1} and evaluates $G_m(k)$.
\begin{lemma}
\label{lem1}
   If $(k_1,k_2)=1$ then $G_m(k_1k_2)=G_m(k_1)G_m(k_2)$. Suppose that $p^a$ is
   the largest power of $p$ dividing $m$ (put $a=\infty$ if $m=0$).
   Then for $b \geq 1$ we have
\begin{equation*}
\label{011}
    G_m(p^b)= \left\{\begin{array}{cl}
    0  & \mbox{if $b\leq a$ is odd}, \\
    \varphi(p^b) & \mbox{if $b\leq a$ is even},  \\
    -p^a  & \mbox{if $b=a+1$ is even}, \\
    (\frac {m/p^a}{p})p^a\sqrt{p}  & \mbox{if $b=a+1$ is odd}, \\
    0  & \mbox{if $b \geq a+2$}.
    \end{array}\right.
\end{equation*}
\end{lemma}

\subsection{Poisson Summation}
    For any smooth function $F$, we write $\hat{F}$ for the Fourier transform of $F$ and we define
\begin{equation} \label{tildedef}
   \widetilde{F}(\xi)=\frac {1+i}{2}\hat{F}(\xi)+\frac
   {1-i}{2}\hat{F}(-\xi)=\int\limits^{\infty}_{-\infty}\left(\cos(2\pi \xi
   x)+\sin(2\pi \xi x) \right)F(x) \dif x.
\end{equation}

    We note the following Poisson summation formula from \cite[Lemma 2.6]{sound1}.
\begin{lemma}
\label{lem2}
   Let $W$ be a smooth function compactly supported on ${\mr}_+$. We have, for any odd integer $n$,
\begin{equation*}
\label{013}
  \sum_{(d,2)=1}\left( \frac {d}{n} \right)
    W\left( \frac {d}{X} \right)=\frac {X}{2n}\left( \frac {2}{n} \right)
    \sum_k(-1)^kG_k(n)\widetilde{W}\left( \frac {kX}{2n} \right),
\end{equation*}
where $\widetilde{W}$ is defined in \eqref{tildedef} and $G_k(n)$ is defined in \eqref{G}.
\end{lemma}

\subsection{Upper bounds for $|L(s, \chi)|^{-1}$}

From \cite[Theorem 5.19]{iwakow}, we deduce the following.
\begin{lemma}
\label{lem:Linversebound}
 Assume the truth of GRH. For any Dirichlet character $\chi$ modulo $q$ and any $\varepsilon > 0$, we have
\begin{equation*}
\label{eq:H-B}
 \big|L(\tfrac 12+\varepsilon+it, \chi)\big|^{-1} \ll \big(q(1+|t|)\big)^{\varepsilon},
\end{equation*}
  where the implied constant depends on $\varepsilon$ alone.
\end{lemma}

\subsection{Analytical behaviors of some Dirichlet Series}
   We define for any square-free $k_1$,
\begin{equation}
\label{eq:Z}
 Z(\alpha,\beta,\gamma;q,k_1) = \sum_{k_2=1}^{\infty} \sum_{(n_1,2q)=1} \sum_{(n_2,2q)=1} \frac{\mu(n_1)\mu(n_2)}{n_1^{\alpha} n_2^{\beta} k_2^{2\gamma}} \frac{G_{k_1 k_2^2}(n_1 n_2)}{n_1 n_2},
\end{equation}
 where $G_m(k)$ be defined as in \eqref{G}. Note first that Lemma \ref{lem1} implies that $Z(\alpha,\beta,\gamma;q,k_1)$ converges absolutely when $\Re(\alpha)$, $\Re(\beta)$, and $\Re(\gamma)$ are all strictly greater than $\tfrac 12$. We write $L_c(s, \chi)$ for the Euler product of $L(s, \chi)$ with the factors from $p | c$ removed.  Our next lemma describes the analytical behavior of $Z$.
\begin{lemma}
\label{lemma:Z}
 The function $Z(\alpha,\beta,\gamma;q,k_1)$ defined in \eqref{eq:Z} may be written as
 \begin{align}
\label{Z2def}
 L_q(\tfrac 12+\alpha, \chi_{k_1})^{-1}L_q(\tfrac 12+\beta,\chi_{k_1})^{-1}Z_{2}(\alpha,\beta,\gamma;q, k_1),
\end{align}
where $Z_{2}(\alpha,\beta,\gamma;q,k_1)$ is a function uniformly bounded in the region
$\Re(\gamma) \ge \frac 12+\varepsilon$, $\Re(\alpha), \Re(\beta) \geq \varepsilon$ for any $\varepsilon >0$.
\end{lemma}
\begin{proof}
We deduce from Lemma \ref{lem1} that the summand in \eqref{eq:Z} is jointly multiplicative in terms of $n_1, n_2$, and $k_2$, so that we can express $Z(\alpha,\beta,\gamma;q, k_1)$ as an Euler product over all primes $p$. It suffices to match each Euler factor at $p$ for  $Z(\alpha,\beta,\gamma;q,k_1)$ with the corresponding factor in \eqref{Z2def}. \newline

 The contribution of such an Euler factor for the generic case with $p \nmid 2q k_1$ is
\begin{equation*}
 \sum_{k_2, n_1, n_2} \frac{\mu(n_1)\mu(n_2)}{p^{n_1 \alpha + n_2 \beta +2k_2 \gamma}}\frac{ G_{k_1 p^{2k_2}}(p^{n_1 + n_2})}{p^{n_1+n_2}  }.
\end{equation*}
If $\Re(\gamma) \ge \tfrac 12+\varepsilon$, $\Re(\alpha)$, $\Re(\beta)\ge \varepsilon$,
Lemma \ref{lem1} implies that the contribution from the terms $k_2\ge 1$ is
$\ll 1/p^{1+2\varepsilon}$ and the contribution of the
term $k_2=0$ is $1-\chi_{k_1}(p) (p^{-1/2-\alpha}+p^{-1/2-\beta})$.  This calculation readily implies that this Euler factor for $Z$ matches the corresponding one in \eqref{Z2def}. \newline

Similarly,  when $\Re(\gamma)\ge \frac 12+\varepsilon$ and $\Re(\alpha)$, $\Re(\beta)\ge \varepsilon$, Lemma \ref{lem1} implies that
the Euler factor for $p | k_1$ but  $p \nmid 2q$ equals
$$
1-\frac{1}{p^{1+\alpha+\beta}}
 + O\left( \frac{1}{p^{1+\varepsilon}}\right)=1+O\left( \frac{1}{p^{1+\varepsilon}}\right).
$$

Lastly, the corresponding Euler factor for the case $p|2q$ is $(1-p^{-2\gamma})^{-1}= 1+O(1/p^{1+2\varepsilon})$.
The assertion of the lemma now follows from these computations.
\end{proof}

\section{Proof of Theorem \ref{meansquare}}
\label{sec 3}

\subsection{Decomposition of ${\mathcal S}(X,Y; \Phi, W)$}
\label{section:mainprop}

  Expanding the square in \eqref{SXYPW} allows us to recast ${\mathcal S}(X,Y; \Phi, W)$ as
\[
S(h):= \sumstar_{d} \sum_{n_1} \sum_{n_2} \chi_{8d}(n_1n_2)\mu(n_1)\mu(n_2) h(d,n_1,n_2),
\]
where $h(x,y,z)=W(\frac xX)\Phi \left(\frac {y}{Y} \right )\Phi \left(\frac {z}{Y} \right )$ is a smooth function on ${\mr}_+^3$.
We apply the M\"{o}bius inversion to remove the square-free condition on $d$ to obtain that, for an appropriate parameter $Z$ to be chosen later,
\begin{eqnarray*}
 S(h) = \Big(\sum_{\substack{a \leq Z \\ (a,2)=1}} + \sum_{\substack{a > Z \\ (a,2)=1}} \Big) \mu(a)  \sum_{(d,2)=1} \sum_{(n_1,a)=1} \sum_{(n_2,a)=1} \chi_{8d}(n_1n_2)\mu(n_1)\mu(n_2)h(da^2, n_1, n_2)
=:S_1(h)+ S_2(h), \quad \mbox{say}.
\end{eqnarray*}

\subsection{Estimating $S_2(h)$}
\label{section:S2}

  We first estimate $S_2(h)$.  To this end, writing $d=b^2 \ell$ with $\ell$ square-free, and grouping terms according to $c=ab$, we deduce
\begin{equation}
\label{eq:S21}
 S_2(h) = \sum_{(c,2)=1} \sum_{\substack{a > Z \\ a|c}} \mu(a)
 \sumstar_{\ell} \sum_{(n_1,c)=1} \sum_{(n_2,c)=1}
 \chi_{8\ell}(n_1 n_2)\mu(n_1)\mu(n_2) h(c^2 \ell, n_1, n_2).
\end{equation}

 Applying Mellin transforms in the variables $n_1$ and $n_2$ yeilds that the inner triple sum over $\ell$, $n_1$, and $n_2$ in \eqref{eq:S21} is
\begin{equation}
\label{eq:S22}
\frac{1}{(2\pi i)^2} \int\limits_{(1+\varepsilon)} \int\limits_{(1+\varepsilon)}
\sumstar_{\ell} {\check h}(c^2 \ell; u, v) \sum_{\substack{n_1, n_2 \\ (n_1n_2, c)=1}}
\frac{\chi_{8\ell}(n_1) \chi_{8\ell}(n_2)\mu(n_1)\mu(n_2) }{n_1^{u}n_2^{v} } \dif u \, \dif v,
\end{equation}
where
\begin{equation*}
{\check h}(x;u,v) = \int\limits_0^{\infty} \int\limits_0^{\infty} h(x,y,z) y^u z^v \frac{\dif y}{y} \frac{\dif z}{z}.
\end{equation*}

Now integration by parts gives that for $\Re(u)$, $\Re(v) >0$ and any positive integers $A_j$, $1 \leq j \leq 3$,
\begin{equation}
\label{eq:3.12}
{\check h}(x;u,v) \ll \left( 1 + \frac{x}{X} \right)^{-A_1} \frac{Y^{\Re(u)+\Re(v)}}{|uv|(1+|u|)^{A_2} (1 + |v|)^{A_3}}.
\end{equation}
  Note that the sum over $n_1$ and $n_2$ in \eqref{eq:S22} equals $L^{-1}_c(u,\chi_{8\ell}) L^{-1}_c(v, \chi_{8\ell})$ and we can thus
move the lines of integration in \eqref{eq:S22}
to $\Re(u)=\Re(v)=1/2+\varepsilon$ without encountering any poles under GRH.  Moreover,
\begin{align}
\label{Linversebound}
|L^{-1}_c(u, \chi_{8\ell})L^{-1}_c(v,\chi_{8\ell})|
\le d(c)^2 ( |L^{-1}(u, \chi_{8\ell})|^2 + |L^{-1}(v, \chi_{8\ell})|^2),
\end{align}
 where $d(c)$ denotes the value of the divisor function at $c$. \newline

  We now apply \eqref{eq:3.12} with $A_2=A_3=1$ and $A_1$ sufficiently large and Lemma \ref{lem:Linversebound} to get that the expression in \eqref{eq:S22} is
\begin{equation*}
%%\label{eq:S23}
 \ll d(c)^2 Y^{1+2\varepsilon} \int\limits_{-\infty}^{\infty} (1+|t|)^{-2} \sumstar_{\ell} \left(1 + \frac{c^2 \ell}{X}\right)^{-A_1}
\Big|L(\tfrac 12+\varepsilon +it,  \chi_{8\ell})\Big|^{-2}  \ \dif t \ll d(c)^2 (XY)^{1+\varepsilon} /c^2.
\end{equation*}

 We conclude from the above estimation and \eqref{eq:S21} that
\begin{align}
\label{S2}
S_2(h) \ll (XY)^{1 + \varepsilon} Z^{-1+ \varepsilon}.
\end{align}

\subsection{Estimating $S_1(h)$, the main term}
  We evaluate $S_1(h)$ now.  Write for brevity $C = \cos$ and $S = \sin$.  We then apply the Poisson summation formula, Lemma \ref{lem2}, to deduce that
\begin{align}
\label{eq:S1}
 S_1(h)= \frac{X}{2} \sum_{\substack{a \leq Z \\ (a,2)=1}} \frac{\mu(a)}{a^2} \sum_{k \in \mz}
\sum_{(n_1,2a)=1} \sum_{(n_2,2a)=1} \frac{(-1)^kG_k(n_1 n_2)\mu(n_1)\mu(n_2)}{n_1 n_2}\int\limits_0^{\infty} h(xX, n_1, n_2) (C + S)\leg{2\pi k xX}{2n_1 n_2 a^2} \dif x.
\end{align}

Let $S_{1,0}(h)$ for the terms in \eqref{eq:S1} with $k=0$.  Note that
$$
\sum_{\substack{a \leq Z \\ (a,2n_1n_2)=1}} \frac{\mu(a)}{a^2} = \frac{1}{\zeta(2)}
\prod_{p|2n_1n_2} \left( 1-\frac{1}{p^2}\right)^{-1} +O(Z^{-1}) = \frac{8}{\pi^2}\prod_{p|n_1n_2} \left(1-\frac{1}{p^2}\right)^{-1}+ O(Z^{-1}) .
$$
 Moreover, with $\square$ denoting a perfect square, Lemma \ref{lem1} implies that $G_0(m) = \varphi(m)$ if $m = \square$, and is zero otherwise.  Thus, upon setting $h_1(y,z) = \int_{\rear_+} h(xX,y,z) \ \dif x$, we infer that
\begin{align} \label{S10first}
 S_{1,0}(h) =  \frac{4X}{\pi^2}
\sum_{\substack{(n_1 n_2,2)=1 \\ n_1 n_2 = \square}}\mu(n_1)\mu(n_2) \prod_{p|n_1n_2} \left( \frac{p}{p+1}\right)
h_1\left( n_1, n_2\right)
 + O \Big( \frac XZ \sum_{\substack{(n_1 n_2,2)=1 \\ n_1 n_2 = \square}}
 \Big|\mu(n_1)\mu(n_2)h_1(n_1,n_2) \Big |\Big).
\end{align}

Mark that the definition of $h$ implies that $h_1 \ll 1$ and $h_1=0$ unless both $n_1$ and $n_2$ are $\ll Y$. Furthermore, if $n_1$, $n_2$ are square-free, then $n_1n_2=\square$ implies that $n_1=n_2$.  Consequently, the sum in the $O$-term in \eqref{S10first} is $\ll Y$ and
\begin{align*}
%%\label{eq:mainterm}
S_{1,0}(h) = \frac{4X}{\pi^2} \sum_{\substack{(n,2)=1}} \prod_{p|n} \left( \frac{p}{p+1}\right)
\mu^2(n) h_1\left( n, n \right) + O\left( \frac {XY}Z \right).
\end{align*}

  We now apply the Mellin transform to recast $h_1(n,n)$ as
\[  h_1(n,n) = \frac{1}{2\pi i} \int\limits_{(2)}  \frac{Y^u}{n^u}\widetilde{h}_1(u,u) \dif u, \quad \mbox{where} \quad \widetilde{h}_1(u,u) = \int\limits_{\mr_{+}} h_1(yY,yY) y^{u} \frac{\dif y}{y}. \]

   Similar to \eqref{eq:3.12}, we have that for $\Re(u)>0$ and any integer $B \geq 0$,
\begin{equation}
\label{h1bound}
 \widetilde{h}_1(u,u) \ll  \frac{1}{|u|(1+|u|)^{B}}.
\end{equation}
Now we can rewrite $S_{1,0}$ as
\begin{equation}
\label{eqn:4.4}
S_{1,0}(h)= \frac{4X}{\pi^2} \frac{1}{2\pi i} \int\limits_{(2)}  Y^u \widetilde{h}_1(u,u) Z(u)
\dif u + O\left( \frac {XY}Z  \right), \quad \mbox{where} \quad Z(u)  = \sum_{\substack{(n,2)=1 }} \frac{\mu^2(n)}{n^{u}}\prod_{p|n} \left( \frac{p}{p+1}\right).
\end{equation}
  We compute the Euler factors of $Z(u)$ to see that
 \begin{align}
 \label{eq:Z(u,v)}
  & Z(u)=\prod_{p >2} \left( 1+ \frac p{p+1} \cdot \frac {1}{p^{u}} \right) =: \zeta(u)Z_2(u),
 \end{align}
 where $Z_2(u)$ converges absolutely in the region $\Re(u) \geq \frac 12+\varepsilon$ for any $\varepsilon>0$. \newline

 Moving the line of integration in \eqref{eqn:4.4} to $\Re(u) = \frac 12+\varepsilon$, we encounter a simple pole at $u=1$ whose residue gives rise to the main term
 $$
 \frac{4}{\pi^2}  XY \widetilde{h}_1(1,1) Z_2(1).
 $$
Now to estimate the integral on the $\frac 12+\varepsilon$ line, we apply the functional equation for $\zeta(s)$ (see \cite[\S 8]{Da}) and Stirling's formula, together with the convexity bound for $\zeta(s)$, rendering
\begin{align*}
%%\label{zetabound}
   \zeta(s) \ll \begin{cases}
   1 \qquad & \Re(s) >1,\\
   (1+|s|)^{(1-\Re(s))/2} \qquad & 0< \Re(s) <1,\\
    (1+|s|)^{1/2-\Re(s)} \qquad & \Re(s) \leq 0.
\end{cases}
\end{align*}
The above and \eqref{h1bound} with $B=1$ enable us to gather that the integral on the $\frac 12+\varepsilon$ line contributes $\ll XY^{1/2+\varepsilon}$.  One can easily check here that the Lindel\"of hypothesis, a consequence of GRH whose truth we assume, does not lead to a better bound.  Now the above discussion, together with \eqref{eqn:4.4}, implies that
\begin{align}
\label{S10}
 S_{1,0}(h) = \frac{4}{\pi^2}  XY \widetilde{h}_1(1,1) Z_2(1)  + O\left( \frac {XY}Z+XY^{1/2+\varepsilon}   \right).
\end{align}
Here we note that
\begin{align}
\label{h1half}
  \widetilde{h}_1(1,1)=\int\limits_{\mr}W(x) \dif x \left (\int\limits_{\mr} \Phi (y) \dif y \right )^2.
\end{align}

\subsection{Estimating $S_1(h)$, the $k \neq 0$ terms}
\label{section:3.3}
Let $S_3(h)$ denote the contribution to $S_1(h)$ from the terms with $k \neq 0$ in \eqref{eq:S1}. Let $f$ be a smooth function on $\mr_+$ with
rapid decay at infinity and $f$ itself and all its derivatives have finite limits as $x\to 0^+$.  We consider the transform given by
\begin{equation*}
 \widehat{f}_{CS}(y) := \int\limits_0^{\infty} f(x) CS(2\pi xy) \dif x,
\end{equation*}
where $CS$ stands for either the cosine or the sine function.  It is shown in \cite[Sec. 3.3]{S&Y} that
\begin{equation*}
 \widehat{f}_{CS}(y)  = \frac{1}{2\pi i} \int\limits_{(1/2)} \widetilde{f}(1-s) \Gamma(s) CS\left(\frac{\text{sgn}(y) \pi s}{2}\right) (2\pi |y|)^{-s} \dif s.
\end{equation*}
Applying the above transform, we deduce that
\begin{align}
\label{eq:3.30}
\begin{split}
 \int\limits_0^{\infty} h \left(Xx, n_1, n_2 \right) (C + S)\leg{2\pi k xX}{2n_1 n_2 a^2} \dif x =  \frac{1}{2\pi i X}  \int\limits_{(\varepsilon)} \check{h}\left(1-s; n_1, n_2 \right) \leg{n_1 n_2 a^2}{\pi |k| }^{s} \Gamma(s) (C + \text{sgn}(k)S)\left(\frac{\pi s}{2} \right) \dif s,
\end{split}
\end{align}
where
\begin{equation*}
\label{eq:3.31}
 \check{h}(s;y,z) = \int\limits_0^{\infty} h(x,y,z) x^s \frac{\dif x}{x}. %=   \int_0^{\infty} h(x,y,z) x^{s} \frac{dx}{x}.
\end{equation*}

Taking the Mellin transforms in the variables $n_1$ and $n_2$, the right-hand side of \eqref{eq:3.30} equals
\begin{equation*}
\frac 1X \leg{1}{2\pi i}^3  \int\limits_{(1)} \int\limits_{(1)} \int\limits_{(\varepsilon)} \widetilde{h}\left(1-s,u,v \right) \frac{1}{n_1^{u} n_2^{v}} \leg{n_1 n_2 a^2}{\pi |k| }^{s} \Gamma(s) (C + \text{sgn}(k)S)\left(\frac{\pi s}{2} \right) \dif s\, \dif u \, \dif v,
\end{equation*}
  where
\begin{equation*}
\widetilde{h}(s,u,v) = \int_{\mr_{+}^{3}} h(x,y,z) x^{s} y^{u} z^{v} \frac{\dif x}{x} \frac{\dif y}{y} \frac{\dif z}{ z}.
\end{equation*}

Integrating by parts implies that for $\Re(s)$, $\Re(u)$, $\Re(v) >0$ and any integers $E_j \geq 0$, $1 \leq j \leq 3$,
\begin{equation}
\label{eq:h}
|\widetilde {h}(s,u,v)| \ll \frac{X^{\Re(s)} Y^{\Re(u)+\Re(v)}}{|uvs| (1+|s|)^{E_1} (1+|u|)^{E_2} (1 + |v|)^{E_3}}.
\end{equation}

  Applying the above bound in \eqref{eq:S1} leads to
\begin{equation} \label{eq:S31}
\begin{split}
 S_3(h) = \frac{1}{2} \sum_{\substack{a \leq Z \\ (a,2)=1}} & \frac{\mu(a)}{a^2} \sum_{k \neq 0}
\sum_{(n_1,2a)=1} \sum_{(n_2,2a)=1}
\frac{(-1)^kG_{k}(n_1 n_2)\mu(n_1)\mu(n_2)}{n_1 n_2} \\
& \times \leg{1}{2\pi i}^3  \int\limits_{(\varepsilon)} \int\limits_{(\varepsilon)} \int\limits_{(\varepsilon)} \widetilde{h}\left(1-s,u,v \right) \frac{1}{n_1^{u} n_2^{v}} \leg{n_1 n_2 a^2}{\pi |k|  }^{s} \Gamma(s) (C + \text{sgn}(k)S)\left(\frac{\pi s}{2} \right) \dif s \, \dif u \, \dif v.
\end{split}
\end{equation}

 Note that by \eqref{eq:h} and the estimation (see \cite[p. 1107]{S&Y}),
\begin{equation}
\label{Gammabound}
 \Big| \Gamma(s) (C \pm S)\Big( \frac {\pi s}{2} \Big) \Big| \ll |s|^{\Re(s)-1/2},
\end{equation}
 the integral over $s$ in \eqref{eq:S31} may be taken over any vertical lines between $0$ and $1$ and
the integrals over $u, v$ in \eqref{eq:S31} may be taken over any vertical lines between $0$ and $2$.  \newline

Hence we arrive at
\begin{equation}
\label{eq:S31-1}
\begin{split}
 S_3(h) = \frac{1}{2} \sum_{\substack{a \leq Z \\ (a,2)=1}} & \frac{\mu(a)}{a^2}\leg{1}{2\pi i}^3  \int\limits_{(\varepsilon)} \int\limits_{(\varepsilon)} \int\limits_{(\varepsilon)} \widetilde{h}\left(1-s,u,v \right) \sum_{k \neq 0} \sum_{(n_1,2a)=1} \\
 & \times
\sum_{(n_2,2a)=1}
\frac{(-1)^kG_{k}(n_1 n_2)\mu(n_1)\mu(n_2)}{n_1 n_2}  \frac{1}{n_1^{u} n_2^{v}} \leg{n_1 n_2 a^2}{\pi |k|  }^{s} \Gamma(s) (C + \text{sgn}(k)S)\left(\frac{\pi s}{2} \right) \dif s \, \dif u \, \dif v.
\end{split}
\end{equation}

   Now, we write $k = \iota k_1k^2_2$ with $\iota \in \{ \pm 1 \}$ and $k_1>0$ square-free.  We write $f(k)=G_{\iota k}(n_1n_2)/|k|^{s}$. It follows from  \cite[(5.15)]{Young2} that
\begin{align*}
   \sum^{\infty}_{k=1}(-1)^{k} f(k)= (2^{1-2s}-1)\sumstar_{\substack{k_1  \geq 1 \\ (k_1, 2) = 1}}
   \sum^{\infty}_{k_2=1}f(k_1k^2_2) +\sumstar_{\substack{k_1 \geq 1 \\ 2| k_1}}\sum^{\infty}_{k_2=1}f(k_1k^2_2).
\end{align*}

    We apply the above relation to recast the expression given in \eqref{eq:S31-1} for $S_3(h)$ as
\begin{align}
\label{S3decomp}
     S_3(h) &=   \sum_{\iota =\pm 1} (S^{\iota}_{3,1}(h)+S^{\iota}_{3,2}(h)),
\end{align}
  where
\begin{align*}
%%\label{eq:S3}
 S^{\iota}_{3,1}(h) =&  \frac{1}{2} \sum_{\substack{a \leq Z \\ (a,2)=1}} \frac{\mu(a)}{a^2} \sumstar_{\substack{k_1  \geq 1 \\ (k_1, 2) = 1}} \leg{1}{2\pi i}^3
\\
& \times  \int\limits_{(1+2\varepsilon)} \int\limits_{(1+2\varepsilon)} \int\limits_{(1/2+\varepsilon)} \widetilde{h}\left(1-s,u,v \right) \Gamma(s) (2^{1-2s}-1) (C + \iota S)\left(\frac{\pi s}{2} \right)\leg{a^2}{\pi k_1  }^{s}Z(u-s, v-s, s;a, \iota k_1) \dif s \, \dif u \, \dif v, \\
S^{\iota}_{3,2}(h) =&  \frac{1}{2} \sum_{\substack{a \leq Z \\ (a,2)=1}} \frac{\mu(a)}{a^2} \sumstar_{\substack{k_1 \geq 1 \\ 2| k_1}} \leg{1}{2\pi i}^3
\\
&  \times \int\limits_{(1+2\varepsilon)} \int\limits_{(1+2\varepsilon)} \int\limits_{(1/2+\varepsilon)} \widetilde{h}\left(1-s,u,v \right)  \Gamma(s)  (C + \iota S)\left(\frac{\pi s}{2} \right)\leg{a^2}{\pi k_1  }^{s}Z(u-s, v-s, s;a, \iota k_1) \dif s \, \dif u \, \dif v.
\end{align*}
 Here the function $Z$ is defined in \eqref{eq:Z}.  We make a change of variables to rewrite $S^{\iota}_{3,1}(h)$ as
\begin{multline*}
S^{\iota}_{3,1}(h)  = \frac{1}{2} \sum_{\substack{a \leq Z \\ (a,2)=1}} \frac{\mu(a)}{a^2} \sumstar_{\substack{k_1  \geq 1 \\ (k_1, 2) = 1}}
\left(\frac{1}{2\pi i}\right)^3  \int\limits_{(1/2+\varepsilon)} \int\limits_{(1/2+\varepsilon)}
\int\limits_{(1/2+\varepsilon)} {\widetilde h}(1-s,u+s,v+s)\Gamma(s) (2^{1-2s}-1)
\\
\hskip 1in \times   (C + \iota S) \left( \frac{\pi s}{2}\right)
 \left(\frac{a^2}{\pi k_1}\right)^s
Z(u, v, s;a, \iota k_1) \dif s \, \dif u \, \dif v.
\end{multline*}

  We split the sum over $k_1$ into two terms according to whether $k_1 \le K$ or not, with $K$ to be optimized later.  If $k_1 \le K$, we move the lines of integration to $\Re(s)= c_1$ for some $1/2<c_1<1$, $\Re(u)=\Re(v)=\varepsilon$. Otherwise, we move
the lines of integration to $\Re(s)=c_2$ for some $c_2>1$, $\Re(u)=\Re(v)=\varepsilon$.
We encounter no poles in either case. Applying Lemma \ref{lemma:Z} and the bound in \eqref{Linversebound} yields
\begin{equation*}
%%\label{eq:Zbound}
Z(u, v,s;a, \iota k_1)
\ll |L^{-1}_a(\tfrac 12+u, \chi_{\iota k_1}) L^{-1}_a(\tfrac 12+v, \chi_{\iota k_1})| \ll d^2 (a) \left(|L^{-1}(\tfrac 12+u, \chi_{\iota k_1})|^2 +
|L^{-1}(\tfrac 12+v, \chi_{\iota k_1})|^2\right).
\end{equation*}
The above and \eqref{eq:h} with $E_1=E_2=E_3=1$, together with \eqref{Gammabound} and the symmetry in $u$ and $v$ give that the terms with $k_1 \le K$ contribute
\begin{equation} \label{eq:firstbd}
\begin{split}
\ll X^{1-c_1} & Y^{2c_1+2\varepsilon}  \sum_{a\le Z} \frac{d(a)}{a^{2-2c_1}}  \\
& \int\limits_{(c_1)} \int\limits_{(\varepsilon)}\int\limits_{(\varepsilon)}  \sumstar_{1 \leq k_1\le K}\frac{1}{k_1^{c_1}} |L(\tfrac 12+u, \chi_{\iota k_1})|^{-2} \frac{ |s|^{\Re(s)-1/2} \dif u \, \dif v \, \dif s}{|1-s|(1+|1-s|)|u+s|(1+|u+s|)|v+s|(1+|v+s|)}.
\end{split}
\end{equation}

  We further apply Lemma \ref{lem:Linversebound} to get
\begin{align*}
 \sumstar_{1 \leq k_1 \le K}\frac{1}{k_1^{c_1}} |L(\tfrac 12+u, \chi_{\iota k_1})|^{-2} \ll K^{1-c_1+\varepsilon}(1+|t|)^{\varepsilon} \ll K^{1-c_1+\varepsilon}\left ((1+|u+s|)^{\varepsilon}+|s|^{\varepsilon} \right ).
\end{align*}

Applying the above in \eqref{eq:firstbd}, we infer that the terms with $k_1 \le K$ contribute
\begin{align*}
%%\label{eq:firstbd1}
\ll X^{1-c_1} Y^{2c_2+2\varepsilon} K^{1-c_1+\varepsilon}Z^{2c_1-1+\varepsilon}.
\end{align*}

Similarly, the contribution from the complementary terms with $k_1 >K$  is
\begin{align*}
%%\label{eq:secbd}
\ll X^{1-c_2} Y^{2c_2+\varepsilon} K^{1-c_2+\varepsilon}Z^{2c_2-1+\varepsilon}.
\end{align*}
  We now balance these contributions by setting $K=Y^2Z^2/X$ so that
\begin{align*}
 X^{1-c_1} Y^{2c_1} K^{1-c_1}Z^{2c_1-1}= X^{1-c_2} Y^{2c_2} K^{1-c_2}Z^{2c_2-1}.
\end{align*}
Now taking $c_1=1/2+\varepsilon$ yields the bound
\begin{align*}
S^{\iota}_{3,1}(h) \ll (XYZ)^{\varepsilon}Y^2Z.
\end{align*}

  Note that $S^{\iota}_{3,2}(h)$ satisfies the above upper bound as well.
Hence, we conclude from \eqref{S2}, \eqref{S10}, \eqref{S3decomp} and the above that
\begin{align*}
%%\label{S}
 S(h) = \frac{4}{\pi^2}  XY \widetilde{h}_1(1,1) Z_2(1)  + O\left( (XY)^{1 + \varepsilon} Z^{-1+ \varepsilon}+XY^{1/2+\varepsilon}+(XYZ)^{\varepsilon}Y^2Z   \right).
\end{align*}

Now \eqref{S} follows upon setting $Z=(X/Y)^{1/2}$, completing the proof of Theorem \ref{meansquare}.

\vspace*{.5cm}

\noindent{\bf Acknowledgments.}   P. G. was supported in part by NSFC Grant 11871082 and L. Z. by the Faculty Silverstar Grant PS65447 at the University of New South Wales.  The authors would also like to thank the anonymous referee for his/her careful and prompt inspection of the paper.

\bibliography{biblio}

% \bib, bibdiv, biblist are defined by the amsrefs package.
\begin{bibdiv}
\begin{biblist}

\bib{BR08}{article}{
      author={Balazard, M.},
      author={de~Roton, A.},
       title={Notes de lecture de l'article $\ll$ partial sums of the m\"obius
  function $\gg$ de {K}annan {S}oundararajan},
        date={Preprint},
        note={arXiv:0810.3587},
}

\bib{CFS}{article}{
      author={Conrey, J.~B.},
      author={Farmer, D.~W.},
      author={Soundararajan, K.},
       title={Transition mean values of real characters},
        date={2000},
     journal={J. Number Theory},
      volume={82},
      number={1},
       pages={109\ndash 120},
}

\bib{Da}{book}{
      author={Davenport, H.},
       title={{M}ultiplicative {N}umber {T}heory},
     edition={Third edition},
      series={Graduate Texts in Mathematics},
   publisher={Springer-Verlag},
     address={Berlin, etc.},
        date={2000},
      volume={74},
}

\bib{G&Zhao2019}{article}{
      author={Gao, P.},
      author={Zhao, L.},
       title={Mean values of some {H}ecke characters},
        date={2019},
     journal={Acta Arith.},
      volume={187},
      number={2},
       pages={125\ndash 141},
}

\bib{G&Zhao2020}{article}{
      author={Gao, P.},
      author={Zhao, L.},
       title={Mean values of cubic and quartic {D}irichlet characters},
        date={2020},
     journal={Funct. Approx. Comment. Math.},
      volume={63},
      number={2},
       pages={227\ndash 245},
}

\bib{G&Zhao2022}{article}{
      author={Gao, P.},
      author={Zhao, L.},
       title={Average values of quadratic {H}ecke character sums},
        date={2022},
     journal={J. Ramanujan Math. Soc.},
      volume={37},
      number={2},
       pages={99\ndash 107},
}

\bib{H&S13}{article}{
      author={Halupczok, K.},
      author={Suger, B.},
       title={Partial sums of the {M}\"{o}bius function in arithmetic
  progressions assuming {GRH}},
        date={2013},
     journal={Funct. Approx. Comment. Math.},
      volume={48},
      number={part 1},
       pages={61\ndash 90},
}

\bib{iwakow}{book}{
      author={Iwaniec, H.},
      author={Kowalski, E.},
       title={{A}nalytic {N}umber {T}heory},
      series={American Mathematical Society Colloquium Publications},
   publisher={American Mathematical Society},
     address={Providence},
        date={2004},
      volume={53},
}

\bib{Landau24}{article}{
      author={Landau, E.},
       title={\"{U}ber die {$\xi$}-{F}unktion und die {$L$}-{F}unktionen},
        date={1924},
     journal={Math. Z.},
      volume={20},
      number={1},
       pages={105\ndash 125},
}

\bib{MM09}{article}{
      author={Maier, H.},
      author={Montgomery, H.~L.},
       title={The sum of the {M}\"{o}bius function},
        date={2009},
     journal={Bull. Lond. Math. Soc.},
      volume={41},
      number={2},
       pages={213\ndash 226},
}

\bib{M&S16}{article}{
      author={Maier, H.},
      author={Sankaranarayanan, A.},
       title={On multiplicative functions resembling the {M}\"{o}bius
  function},
        date={2016},
     journal={J. Number Theory},
      volume={163},
       pages={75\ndash 88},
}

\bib{Ng04}{article}{
      author={Ng, N.},
       title={The distribution of the summatory function of the {M}\"{o}bius
  function},
        date={2004},
     journal={Proc. London Math. Soc. (3)},
      volume={89},
      number={2},
       pages={361\ndash 389},
}

\bib{O&R}{article}{
      author={Odlyzko, A.~M.},
      author={te~Riele, H. J.~J.},
       title={Disproof of the {M}ertens conjecture},
        date={1985},
     journal={J. Reine Angew. Math.},
      volume={357},
       pages={138\ndash 160},
}

\bib{sound1}{article}{
      author={Soundararajan, K.},
       title={Nonvanishing of quadratic {D}irichlet {$L$}-functions at
  $s=\frac{1}{2}$},
        date={2000},
     journal={Ann. of Math. (2)},
      volume={152},
      number={2},
       pages={447\ndash 488},
}

\bib{Sound09-1}{article}{
      author={Soundararajan, K.},
       title={Partial sums of the {M}\"{o}bius function},
        date={2009},
     journal={J. Reine Angew. Math.},
      volume={631},
       pages={141\ndash 152},
}

\bib{S&Y}{article}{
      author={Soundararajan, K.},
      author={Young, M.~P.},
       title={The second moment of quadratic twists of modular
  {$L$}-functions},
        date={2010},
     journal={J. Eur. Math. Soc. (JEMS)},
      volume={12},
      number={5},
       pages={1097\ndash 1116},
}

\bib{Titchmarsh27}{article}{
      author={Titchmarsh, E.~C.},
       title={A {C}onsequence of the {R}iemann {H}ypothesis},
        date={1927},
     journal={J. London Math. Soc.},
      volume={2},
      number={4},
       pages={247\ndash 254},
}

\bib{Ye}{article}{
      author={Ye, L.},
       title={Bounding sums of the {M}\"obius function over arithmetic
  progressions},
        date={Preprint},
        note={arXiv:1406.7326},
}

\bib{Young2}{article}{
      author={Young, M.~P.},
       title={The third moment of quadratic {D}irichlet {L}-functions},
        date={2013},
     journal={Selecta Math. (N.S.)},
      volume={19},
      number={2},
       pages={509\ndash 543},
}

\end{biblist}
\end{bibdiv}
\bibliographystyle{amsxport}

\end{document}